\documentclass[12pt]{amsart}

\usepackage{amsfonts}
\usepackage{latexsym}
\usepackage{amsmath}
\usepackage{amssymb}
\usepackage{amsthm}
\usepackage{enumerate} 
\usepackage{mathtools} 
\usepackage[margin=1.2in]{geometry}
\usepackage{mathabx}

\theoremstyle{plain}
\newtheorem{theorem}{Theorem}
\newtheorem{lemma}{Lemma}
\newtheorem{proposition}{Proposition}
\newtheorem{corollary}{Corollary}

\theoremstyle{definition}

\newtheorem{assumption}{Assumption}

\newtheorem{remark}[theorem]{Remark}

\newcommand\naturals{\mathbb{N}}

\newcommand\reals{\mathbb{R}}
\newcommand\complexes{\mathbb{C}}

\newcommand\norm[1]{\left\lVert#1\right\rVert}
\newcommand\ev[2]{\left<#1,#2\right>}

\DeclareMathOperator{\im}{Im}

\DeclareMathOperator{\id}{id}

\newcommand\es{\mathcal{E}}

\newcommand\sss{\mathcal{S}}

\newcommand\des{\es^{\prime}}

\newcommand\despaceomega[1]{\testspaceomega{\des}}

\newcommand\beurou{\ast}

\numberwithin{equation}{section}

\begin{document}

\title[On the space of ultradistributions vanishing at infinity]{On the space of ultradistributions vanishing at infinity}
\author[A. Debrouwere]{Andreas Debrouwere}
\thanks{A. Debrouwere was supported by  FWO-Vlaanderen via the postdoctoral grant 12T0519N}

\author[L. Neyt]{Lenny Neyt}
\thanks{L. Neyt gratefully acknowledges support by Ghent University through the BOF-grant 01J11615.}

\author[J. Vindas]{Jasson Vindas}
\thanks {The work of J. Vindas was supported by Ghent University through the BOF-grants 01J11615 and 01J04017.}

\address{Department of Mathematics: Analysis, Logic and Discrete Mathematics\\ Ghent University\\ Krijgslaan 281\\ 9000 Gent\\ Belgium}
\email{andreas.debrouwere@UGent.be}
\email{lenny.neyt@UGent.be}
\email{jasson.vindas@UGent.be}

\subjclass[2010]{\emph{Primary.} 46F05. \emph{Secondary.} 42B10,  46F12, 81S30.}
\keywords{The space of  ultradistributions vanishing at infinity; the first structure theorem; S-asymptotics; the short-time Fourier transform}

\begin{abstract}

We study the structural and linear topological properties of the space $\dot{\mathcal{B}}^{\prime \beurou}_{\omega}$ of ultradistributions vanishing at infinity (with respect to a weight function $\omega$). Particularly, we show the first structure theorem for $\dot{\mathcal{B}}^{\prime \beurou}_{\omega}$ under weaker hypotheses than were known so far. As an application, we determine the structure of the S-asymptotic behavior of ultradistributions.
\end{abstract}

\maketitle

\section{Introduction}
The space $\mathcal{B}^{\prime}$ of bounded distributions and its subspace $\dot{\mathcal{B}}^{\prime}$ of distributions vanishing at infinity, introduced by Schwartz \cite{Schwartz}, play an important role in the convolution theory for distributions \cite{Ortner-1, Ortner, N-O} and the asymptotic analysis of generalized functions \cite{P-S-V}. 
Their analogues in the setting of ultradistributions were first considered in \cite{cioranescu1992,P-CharBoundedUltradistr} and further studied in  \cite{B-F-G-ultradistrLpgrowth, carmichael2007boundary, dimovski2016, D-P-V2015, N-V-asympboundMAE}. In \cite{dimovski2016}, the second structure theorem for these spaces (and their weighted variants) was shown by means of the parametrix method. This technique imposes heavy restrictions on the defining weight sequence, namely, the assumptions \cite{ultradistributions1} $(M.1)$, $(M.2)$, and $(M.3)$. The last two named authors have recently provided in \cite{N-V-asympboundMAE} the first structure theorem for the space of bounded ultradistributions (with respect to a weight function $\omega$)  under the weaker assumptions $(M.1)$, $(M.2)'$, and $(M.3)'$.

 The main goal of this article is to show the first structure theorem for the space of ultradistributions vanishing at infinity (with respect to a weight function $\omega$). More precisely, we shall prove the following result; we refer to Sections \ref{sect-prelim} and \ref{sect-struct} for the definition of unexplained notions.

\begin{theorem}
	\label{t:structuraltheorem}
	Let $M_{p}$ be a weight sequence satisfying $(M.1)$ and $(M.2)'$, and let $\omega$ be a weight function such that Assumption \ref{assumption} holds (cf. Section \ref{sect-struct}). Then, for every $f \in \dot{\mathcal{B}}^{\prime \ast}_\omega$ there exist $f_{\alpha} \in C(\reals^{d})$, $\alpha \in \naturals^{d}$, such that
	$$
					f = \sum_{\alpha \in \naturals^{d}} f_{\alpha}^{(\alpha)},
	$$			
 the limits
	$$
	\lim_{|x| \rightarrow \infty} \frac{f_{\alpha}(x)}{\omega(x)} = 0 , \qquad \forall \alpha \in \naturals^{d} , 
$$			
	 hold, and  for some $\ell > 0$ (for all $\ell > 0$) we have that
	 $$
					\sup_{\alpha \in \naturals^d} \sup_{x \in \reals^d} \frac{\ell^{|\alpha|}{M_{\alpha} |f_{\alpha}(x)|}}{\omega(x)} < \infty.
$$
				\end{theorem}
If $\omega$ is a weight function satisfying 
$$
\operatorname*{ess \: sup}_{x \in \reals^d} \frac{\omega(\:\cdot\: + x)}{ \omega(x)} \in L^\infty_{\operatorname{loc}},
$$
then Assumption \ref{assumption} holds for $M_p$ and $\omega$, where $M_p$ is any log-convex weight sequence  satisfying $\lim_{p\to\infty} M_p^{-1/p}\log p=0$; see Remark \ref{remark-1}. In particular, in this case the Assumption \ref{assumption} is fulfilled when $p!^{\sigma}\subset M_p$ for some $\sigma>0$.

It is important to point out that none of the methods available in the literature applies to deliver a proof for Theorem \ref{t:structuraltheorem}. We develop here a new approach to the problem whose core consists in combining a criterion for the surjectivity of a continuous linear mapping in terms of its transpose (Lemma \ref{t:surjectivity}) with the computation of the dual of $\dot{\mathcal{B}}'^*_{\omega}$.  The latter computation is achieved in this article by exploiting the mapping properties of the short-time Fourier transform (STFT). In fact, we shall show that the strong dual of $\dot{\mathcal{B}}^{\prime \ast}_{\omega}$ is given by $\mathcal{D}^{\beurou}_{L^{1}_{\omega}}$.  We mention that the STFT has recently proved to be a powerful tool in the study of the structural and linear topological properties of (generalized) function spaces; see \cite{B-O2014, D-V-indlimultra, D-V-ConvSTFT, K-P-S-V2016,  N-VMultDimTaub}.

As an application of Theorem \ref{t:structuraltheorem}, we determine the structure of the S-asymptotic behavior of ultradistributions. Theorem \ref{t:structSasymp} below may be interpreted as the first structure theorem for S-asymptotics, whereas \cite[Theorem 1.10, p.\ 46]{P-S-V} may be seen as the second structure theorem for S-asymptotics. As a consequence, we obtain that all results from \cite{structquasiultra} essentially hold under the weaker assumptions $(M.1)$, $(M.2)'$, and $(M.3)'$ on the defining weight sequence; see Remark \ref{remark-2}.

This paper is organized as follows. In the preliminary Section \ref{sect-prelim}, we  first present a criterion for the surjectivity of a continuous linear mapping in terms of its transpose, after which we introduce Gelfand-Shilov spaces and their duals, and briefly discuss the mapping properties of the STFT on these spaces. In Section \ref{sect-struct}, we define $\mathcal{D}^{\beurou}_{L^{1}_{\omega}}$ and $\dot{\mathcal{B}}^{\prime \beurou}_{\omega}$ and characterize these spaces via the STFT. The equality $(\dot{\mathcal{B}}^{\prime \ast}_\omega)'_b = \mathcal{D}^{\beurou}_{L^{1}_{\omega}}$  and Theorem \ref{t:structuraltheorem} are shown in Section \ref{sect-main}. Finally, in Section \ref{sect-strucS}, we present our results about the S-asymptotic behavior of ultradistributions. 
\section{Preliminaries}\label{sect-prelim}
Given a lcHs (= locally convex Hausdorff space) $E$, we denote its dual by $E'$. Unless explicitly stated otherwise, we endow $E'$ with the strong topology. 
\subsection{Surjections in locally convex spaces} A continuous linear mapping between Fr\'echet spaces is surjective if and only if its transpose is injective and has weakly closed range  \cite[Theorem 37.2, p.\ 382]{Treves}. In the proof of Theorem \ref{t:structuraltheorem}, we will make use of the following generalization of this criterion.

\begin{lemma}
	\label{t:surjectivity}
	Let $E$ and $F$ be lcHs and let $S: E \rightarrow F$ be a continuous linear mapping. Suppose that $E$ is Mackey, $E / \ker S$ is complete, and $\im S$ is Mackey for the topology induced by $F$. Then, $S$ is surjective if the following two conditions are satisfied:
		\begin{enumerate}
			\item $S^{t} : F' \rightarrow E'$ is injective;
			\item $\im S^{t}$ is weakly closed in $E'$.
		\end{enumerate}
\end{lemma}

\begin{proof}
If $S^{t}$ is injective, then $\im S$ is dense in $F$. Hence, it suffices to show that $\im S$ is closed in $F$. As $\im S^{t}$ is weakly closed,  $S$ is a weak  homomorphism \cite[Lemma 37.4]{Treves}. Since $\sigma(E / \ker S, (E / \ker S)')$ coincides with the quotient of $\sigma(E,E')$ modulo $\ker S$ \cite[p.\ 385]{Treves} and $\sigma(\im S, (\im S)')$ coincides with the topology induced by $\sigma(F,F')$, we obtain that $\widetilde{S} : E / \ker S \rightarrow \im S$ is a weak isomorphism. Consequently, $\widetilde{S}$ is also an isomorphism if we equip  $E / \ker S$  and $\im S$ with their Mackey topology \cite[p.\ 158]{S-TVS}. From this we may infer that $S$ is a homomorphism because $E / \ker S$ is Mackey as $E$ is so \cite[p.\ 136]{S-TVS} and $\im S$ is Mackey by assumption. Finally, since $E / \ker S$ is complete, we have that $\im S \cong E / \ker S$ is complete and, thus, closed in $F$.

\end{proof}

\subsection{Gelfand-Shilov spaces and the short-time Fourier transform} A sequence $(M_{p})_{p \in \naturals}$ of positive real numbers is called a \emph{weight sequence} if $M_p/M_{p-1} \to \infty$ as $p \to \infty$. We will make use of some of the following conditions on weight sequences:
	\begin{description}
		\item[$(M.1)$] $M_{p}^{2} \leq M_{p - 1} M_{p + 1}$ , $p \geq 1$; 
		\item[$(M.2)'$] $M_{p + 1} \leq C_0 H^{p} M_{p}$, $p \in \naturals$, for some $C_0,H \geq 1$;
		 \item[$(M.3)'$] $\sum_{p = 1}^{\infty} M_{p-1}/M_p < \infty$.
	\end{description}
The reader is referred to \cite{ultradistributions1} for the meaning of these conditions. For a multi-index $\alpha \in \naturals^{d}$, we simply write $M_{\alpha} = M_{|\alpha|}$. As usual, the relation $M_{p} \subset N_{p}$ between two weight sequences means that there exist $C, \varepsilon > 0$ such that $M_{p} \leq C \varepsilon^{p} N_{p}$ for all $p \in \naturals$. The stronger relation $M_{p} \prec N_{p}$ means that the latter inequality remains valid for all $\varepsilon > 0$ and suitable $C = C_{\varepsilon} > 0$. The \emph{associated function} of $M_{p}$ is defined as
	\[ M(t) = \sup_{p \in \naturals} \log \frac{t^{p} M_{0}}{M_{p}} , \qquad t > 0 , \]
and $M(0) = 0$. We define $M$ on $\reals^{d}$ as the radial function $M(x) = M(|x|)$, $x \in \reals^{d}$. We will often use the following result \cite[Proposition 3.4]{ultradistributions1}: If $M_{p}$ satisfies $(M.1)$ and $(M.2)'$, then, for any $k > 0$,
	\begin{equation}
		\label{eq:M2'}
		M(t) - M(kt) \leq - \frac{\log(t / C_0) \log k}{\log H} , \qquad t > 0 . 
	\end{equation}
	
Let $M_{p}$ and $A_{p}$ be two weight sequences. We denote by $A$ the associated function of $A_{p}$. For $\ell, q > 0$ we denote by $\sss^{M_{p}, \ell}_{A_{p}, q}(\reals^{d})$  the Banach space consisting of all $\varphi \in C^{\infty}(\reals^{d})$ such that
	\[ \norm{\varphi}_{\sss^{M_{p}, \ell}_{A_{p}, q}} := \sup_{\alpha \in \naturals^{d}} \sup_{x \in \reals^{d}} \frac{|\varphi^{(\alpha)}(x)| e^{A(q x)}}{\ell^{|\alpha|} M_{\alpha}} < \infty . \]
We define
	\[ \sss^{(M_{p})}_{(A_{p})}(\reals^{d}) := \varprojlim_{\ell \rightarrow 0^{+}} \sss^{M_{p}, \ell}_{A_{p}, 1 / \ell}(\reals^{d}) , \qquad \sss^{\{M_{p}\}}_{\{A_{p}\}}(\reals^{d}) := \varinjlim_{\ell \rightarrow \infty} \sss^{M_{p}, \ell}_{A_{p}, 1 / \ell}(\reals^{d}) . \]

We shall write $\beurou$ instead of $(M_{p})$ or $\{M_{p}\}$ and $\dagger$ instead of $(A_{p})$ or $\{A_{p}\}$ if we want to treat both cases simultaneously. In addition, we shall often first state assertions for the Beurling case followed in parenthesis by the corresponding statements for the Roumieu case. 

Following \cite{ultradistributions1}, we denote by $\mathcal{E}^\ast(\reals^d)$  the space of ultradifferentiable functions of class $\ast$  and by $\mathcal{D}^\ast(\reals^d)$ the space of compactly supported ultradifferentiable functions of class $\ast$, each of them endowed with their natural locally convex topology. If $M_p$ satisfies $(M.1)$,  $\mathcal{D}^\ast(\reals^d)$ is non-trivial if and only if $M_p$ satisfies $(M.3)'$, as follows from the Denjoy-Carleman theorem. 

Next, we discuss the \emph{short-time Fourier transform} (STFT); see \cite{Grochenig} for an extensive overview. We denote the translation and modulation operators by $T_{x} f(t) = f(t - x)$ and $M_{\xi}f(t)= e^{2 \pi i \xi \cdot t} f(t)$ for $x, \xi \in \reals^{d}$. We also write $\check{f}(t) = f(-t)$ for reflection about the origin. The STFT of a function $f \in L^{2}(\reals^{d})$ with respect to the window $\psi \in L^{2}(\reals^{d})$ is given by
	\[ V_{\psi} f(x, \xi) := (f, M_{\xi} T_{x} \psi)_{L^{2}} = \int_{\reals^{d}} f(t) \overline{\psi(t - x)} e^{- 2 \pi i \xi \cdot t} dt , \qquad (x, \xi) \in \reals^{2d} . \]
It holds that $\norm{V_{\psi}f}_{L^{2}(\reals^{2d})} = \norm{\psi}_{L^{2}} \norm{f}_{L^{2}}$. In particular, $V_{\psi} : L^{2}(\reals^{d}) \rightarrow L^{2}(\reals^{2d})$ is continuous. The adjoint of $V_{\psi}$ is given by the weak integral
	\[ V_{\psi}^{*} F := \int \int_{\reals^{2d}} F(x, \xi) M_{\xi} T_{x} \psi dx d\xi , \qquad F \in L^{2}(\reals^{2d}) . \]
If $\psi \neq 0$ and $\gamma \in L^{2}(\reals^{d})$ is a synthesis window for $\psi$, i.e. $(\gamma, \psi)_{L^{2}} \neq 0$, then
	\begin{equation} 
		\label{eq:reconstructSTFT}
		\frac{1}{(\gamma, \psi)_{L^{2}}} V_{\gamma}^{*} \circ V_{\psi} = \id_{L^{2}(\reals^{d})} . 
	\end{equation}

We now study the mapping properties of the STFT on the spaces $\sss^{(M_{p})}_{(A_{p})}(\reals^{d})$ and $\sss^{\prime (M_{p})}_{(A_{p})}(\reals^{d})$ (cf.\ \cite{D-V-indlimultra}). We need some preparation. 
Given two lcHs $E$ and $F$, we write $E \widehat{\otimes}_{\pi} F$ and $E \widehat{\otimes}_{\varepsilon} F$ for the completion of the tensor product $E \otimes F$ with respect to the projective topology and the $\varepsilon$-topology, respectively. If either $E$ or $F$ is nuclear, we simply write $E \widehat{\otimes} F = E \widehat{\otimes}_{\pi} F = E \widehat{\otimes}_{\varepsilon} F$. Let $N_{p}$ be a weight sequence and denote by $N$ its associated function. We define $\sss_{(N_{p})}(\reals^{d})$ as the Fr\'echet space consisting of all $\varphi \in C^{\infty}(\reals^{d})$ such that
	\[ \norm{\varphi}_{\sss_{N_{p}, k}} := \sup_{|\alpha| \leq k} \sup_{x \in \reals^{d}} |\varphi^{(\alpha)}(x)| e^{N(k x)} < \infty , \qquad \forall k \in \naturals . \]
If $N_p$ satisfies $(M.1)$ and $(M.2)'$, then $\sss_{(N_{p})}(\reals^{d})$ is nuclear, as follows from \cite[p.\ 181]{Gelfand-Shilov3} and \eqref{eq:M2'}. 

Let $M_{p}$ and $A_{p}$ be two weight sequences satisfying $(M.1)$ and $(M.2)'$. We may identify $\sss_{(A_{p})}(\reals^{d}_{x}) \widehat{\otimes} \sss_{(M_{p})}(\reals^{d}_{\xi})$ with the Fr\'echet space consisting of all $\Phi \in C^{\infty}(\reals^{2d}_{x, \xi})$ such that
	\[ |\Phi|_{k} := \max_{|\alpha| \leq k} \max_{|\beta| \leq k} \sup_{(x, \xi) \in \reals^{2d}} \left| \partial^{\alpha}_{x} \partial^{\beta}_{\xi} \Phi(x, \xi) \right| \exp[A(k x) + M(k \xi)] < \infty , \qquad \forall k \in \naturals . \]
The following canonical isomorphism of lcHs holds
	\[ (\sss_{(A_{p})}(\reals^{d}_{x}) \widehat{\otimes} \sss_{(M_{p})}(\reals^{d}_{\xi}))^{\prime} \cong \sss^{\prime}_{(A_{p})}(\reals^{d}_{x}) \widehat{\otimes} \sss^{\prime}_{(M_{p})}(\reals^{d}_{\xi}) . \]
We then have:
	\begin{proposition}
		\label{p:contSTFT}
		Let  $\psi \in \sss^{(M_{p})}_{(A_{p})}(\reals^{d})$. Assume that both weight sequences satisfy $(M.1)$ and $(M.2)'$. The following mappings are continuous,
			\[ V_{\psi} : \sss^{(M_{p})}_{(A_{p})}(\reals^{d}) \rightarrow \sss_{(A_{p})}(\reals^{d}_{x}) \widehat{\otimes} \sss_{(M_{p})}(\reals^{d}_{\xi}) \]
		and
			\[ V_{\psi}^{*} : \sss_{(A_{p})}(\reals^{d}_{x}) \widehat{\otimes} \sss_{(M_{p})}(\reals^{d}_{\xi}) \rightarrow \sss^{(M_{p})}_{(A_{p})}(\reals^{d}) . \]
	\end{proposition}
\begin{proof}
	We first consider $V_{\psi}$. Let $\varphi \in \sss^{(M_{p})}_{(A_{p})}(\reals^{d})$ be arbitrary and fix $k \in \naturals$.  For all $\gamma \in \naturals^d$ we have that
			\begin{align*}
			&\max_{|\alpha|, |\beta| \leq k} \left| \xi^{\gamma} \partial^{\alpha}_{x} \partial^{\beta}_{\xi} V_{\psi} \varphi(x, \xi) \right|e^{A(kx)} \\
			&\leq (2\pi)^{k - |\gamma|} \max_{|\alpha|, |\beta| \leq k} \sum_{\delta \leq \gamma} {\gamma \choose \delta} \int_{\reals^{d}} |(t^{\beta} \varphi(t))^{(\delta)}|e^{A(2kt)} |\psi^{(\alpha + \gamma - \delta)}(x - t)|e^{A(2k(x-t))} dt \\
			&\leq C \max_{|\alpha|, |\beta| \leq k} \norm{t^{\beta} \varphi(t)}_{\sss^{M_{p}, \pi/(k\sqrt{d})}_{A_{p}, 2k}} \norm{\psi^{(\alpha)}}_{\sss^{M_{p}, \pi/(k\sqrt{d})}_{A_{p}, 2kH^{d+1}}}\frac{M_{\gamma}}{(k\sqrt{d})^{|\gamma|}}  \\
			&\leq C' \norm{\varphi}_{\sss^{M_{p}, \pi/(k\sqrt{d})}_{A_{p}, 2kH^k}}  \frac{M_{\gamma}}{(k\sqrt{d})^{|\gamma|}}, 					\end{align*}
so that $|V_{\psi} \varphi|_{k} \leq C'M_0  \norm{\varphi}_{\sss^{M_{p}, \pi/(k\sqrt{d})}_{A_{p}, 2kH^k}}$. Next, we treat $V_{\psi}^{*}$. Take an arbitrary function $\Phi \in \sss_{(A_{p})}(\reals^{d}_{x}) \widehat{\otimes} \sss_{(M_{p})}(\reals^{d}_{\xi})$ and fix $\ell > 0$. Pick $k \in \naturals$ such that $k \ell \geq 4 \pi H^{d + 1}$. For all $\alpha \in \naturals^{d}$ we have that
		\begin{align*}
			& \left| \partial^{\alpha} V_{\psi}^{*} \Phi(t) \right|e^{A(t / \ell)} \\
			&\leq \sum_{\beta \leq \alpha} {\alpha \choose \beta} \int \int_{\reals^{2d}} |\Phi(x, \xi)| |2 \pi \xi|^{|\beta|} e^{A(2x / \ell)} |\psi^{(\alpha - \beta)}(t - x)| e^{A(2(t - x) / \ell)} dx d\xi \\
			&\leq C|\Phi|_k \ell^{|\alpha|} M_{\alpha},
		\end{align*}
whence		 $\| V_{\psi}^{*} \Phi\|_{\mathcal{S}^{M_p,\ell}_{A_p,1/\ell}} \leq C |\Phi|_k$.
\end{proof}	
The STFT of an ultradistribution $f \in \sss^{\prime \beurou}_{\dagger}(\reals^{d})$ with respect to a window function $\psi \in \sss^{(M_{p})}_{(A_{p})}(\reals^{d})$ is defined as
	\[ V_{\psi} f(x, \xi) := \ev{f}{\overline{M_{\xi} T_{x} \psi}} = e^{-2\pi i \xi \cdot x} (f * \check{\overline{\psi}})(x) , \qquad (x, \xi) \in \reals^{2d} . \]
Clearly, $V_{\psi} f$ is a smooth function on $\reals^{2d}$. We define the adjoint STFT of $F \in \sss^{\prime}_{(A_{p})}(\reals^{d}_{x}) \widehat{\otimes} \sss^{\prime}_{(M_{p})}(\reals^{d}_{\xi})$ as
	\[ \ev{V_{\psi}^{*} F}{\varphi} := \ev{F}{\overline{V_{\psi} \overline\varphi}} , \qquad \varphi \in \sss^{(M_{p})}_{(A_{p})}(\reals^{d}) . \]
Notice that $V_{\psi}^{*} F \in \sss^{\prime (M_{p})}_{(A_{p})}(\reals^{d})$ by Proposition \ref{p:contSTFT}. In view of Proposition \ref{p:contSTFT}, similar arguments as in \cite[Section 3]{K-P-S-V2016} yield the following result.

	\begin{proposition}
	Let $\psi \in \sss^{(M_{p})}_{(A_{p})}(\reals^{d})$. Assume that both weight sequences satisfy $(M.1)$ and $(M.2)'$. The following mappings are continuous,
			\[ V_{\psi} : \sss^{\prime (M_{p})}_{(A_{p})}(\reals^{d}) \rightarrow \sss^{\prime}_{(A_{p})}(\reals^{d}_{x}) \widehat{\otimes} \sss^{\prime}_{(M_{p})}(\reals^{d}_{\xi}) \]
		and
			\[ V_{\psi}^{*} : \sss^{\prime}_{(A_{p})}(\reals^{d}_{x}) \widehat{\otimes} \sss^{\prime}_{(M_{p})}(\reals^{d}_{\xi}) \rightarrow \sss^{\prime (M_{p})}_{(A_{p})}(\reals^{d}). \]
		Moreover, if $\psi \neq 0$ and $\gamma \in \sss^{(M_{p})}_{(A_{p})}(\reals^{d})$ is a synthesis window for $\psi$, then the reconstruction formula
			\begin{equation}	
			\label{eq:reconstructSTFT-1}
		 \frac{1}{(\gamma, \psi)_{L^{2}}} V_{\gamma}^{*} \circ V_{\psi} = \id_{\sss^{\prime (M_{p})}_{(A_{p})}(\reals^{d}) } 
		 	\end{equation}
		is valid and the desingularization formula
			\begin{equation}
				\label{eq:desing} 
				\ev{f}{\varphi} = \frac{1}{(\gamma, \psi)_{L^{2}}} \int \int_{\reals^{2d}} V_{\psi} f(x, \xi) V_{\overline{\gamma}} \varphi(x, - \xi) dx d\xi 
			\end{equation}
		holds for all $f \in  \sss^{\prime (M_{p})}_{(A_{p})}(\reals^{d})$ and $\varphi \in  \sss^{(M_{p})}_{(A_{p})}(\reals^{d})$. 
	\end{proposition}

\section{The spaces $\mathcal{D}_{L^{1}_{\omega}}^{\beurou}$ and $\dot{\mathcal{B}}_{\omega}^{\prime \beurou}$}\label{sect-struct}

 A measurable function $\omega: \reals^{d} \rightarrow (0, \infty)$ is called a \emph{weight function} if $\omega$ and $\omega^{-1}$ are locally bounded. Given a weight sequence $A_{p}$,  a weight function $\omega$ is said to be \textit{$(A_{p})$-admissible} (\textit{$\{A_{p}\}$-admissible}) if
	\[ \exists \lambda > 0 ~ (\forall \lambda > 0) ~ \exists C > 0 ~  \forall x, t \in \reals^{d} ~ : ~ \omega(x + t) \leq C \omega(x) e^{A(\lambda t)} . \]
Next, we introduce various function and ultradistribution spaces associated to a weight function $\omega$ (cf.\ \cite{D-P-V2015}). We define $L^{1}_{\omega}(\reals^d)$ as the Banach space consisting of all measurable functions $\varphi$ on $\reals^d$ such that
$$
 \norm{\varphi}_{L^{1}_{\omega}} := \int_{\reals^{d}} |\varphi(x)| \omega(x) dx < \infty.
 $$
Its dual is given by the space $L^{\infty}_{\omega}(\reals^d)$ of all those measurable function $\varphi$ on $\reals^d$ such that 
$$
 \norm{\varphi}_{L^{\infty}_{\omega}} := \operatorname*{ess \: sup}_{x \in \reals^d} \frac{|\varphi(x)|}{ \omega(x)} < \infty.
 $$
We write $\mathcal{D}_{L^{1}_{\omega}}(\reals^d)$  for the space consisting of all $\varphi \in C^\infty(\reals^d)$ such that $\varphi^{(\alpha)} \in L^1_\omega(\reals^d)$ for all $\alpha \in \naturals^d$. Let $M_p$ be a weight sequence.  For  $\ell > 0$ we denote by $\mathcal{D}^{M_{p}, \ell}_{L^{1}_{\omega}}$ the Banach space consisting of all $\varphi \in \mathcal{D}_{L^{1}_{\omega}}(\reals^d)$ such that
	\[ \norm{\varphi}_{\mathcal{D}^{M_{p}, \ell}_{L^{1}_{\omega}}} := \sup_{\alpha \in \naturals^{d}} \frac{\norm{\varphi^{(\alpha)}}_{L^{1}_{\omega}}}{\ell^{|\alpha|} M_{\alpha}} < \infty . \]
We define
	\[ \mathcal{D}^{(M_{p})}_{L^{1}_{\omega}} := \varprojlim_{\ell \rightarrow 0^{+}} \mathcal{D}^{M_{p}, \ell}_{L^{1}_{\omega}} , \qquad \mathcal{D}^{\{M_{p}\}}_{L^{1}_{\omega}} := \varinjlim_{\ell \rightarrow \infty} \mathcal{D}^{M_{p}, \ell}_{L^{1}_{\omega}} . \]
If $M_p$ satisfies $(M.2)'$, a standard argument shows that $\mathcal{D}^{\beurou}_{L^{1}_{\omega}} \subset \mathcal{E}^\ast(\reals^d)$ with continuous inclusion. We introduce the following set of assumptions on a weight sequence $M_p$ and a weight function $\omega$.
\begin{assumption} \label{assumption}
	\textit{There exists a weight sequence $A_p$ satisfying $(M.1)$ and $(M.2)'$ such that  $\omega$ is $(A_{p})$-admissible ($\{A_{p}\}$-admissible) and  $\sss^{(M_{p})}_{(A_{p})}(\reals^{d})$ is non-trivial.}
\end{assumption}

\begin{remark}\label{remark-1}  A sufficient condition for the non-triviality of $\mathcal{S}^{(M_p)}_{(A_p)}(\reals^d)$ is $p!^\sigma \subset M_p$ and $p!^\tau \subset A_p$ for some $\sigma, \tau > 0$ with $\sigma + \tau > 1$ \cite[p.\ 235]{GelfandShilov}. Other non-triviality conditions can be found in \cite{D-V-Nontrivialanalyfunc}. If a weight function $\omega$ is $(p!)$-admissible, then Assumption \ref{assumption} is fulfilled for $M_p$ and $\omega$, whenever $M_p$ is a weight sequence that satisfies $(M.1)$ and $(\log p)^{p} \prec M_p$, as follows from \cite[Proposition 2.7 and Theorem 5.9]{D-V-Nontrivialanalyfunc}. We point out that \cite[Remark 5.3]{D-V-ConvSTFT} $\omega$ is $(p!)$-admissible if and only if $L^1_\omega$ is translation-invariant if and only if
\begin{equation}\label{TIWFeq}
\operatorname*{ess \: sup}_{x \in \reals^d} \frac{\omega(\:\cdot\: + x)}{ \omega(x)} \in L^\infty_{\operatorname{loc}}.
\end{equation}
\end{remark}

In the rest of this section, we fix a weight sequence $M_p$ satisfying $(M.1)$ and $(M.2)'$, and a weight function $\omega$ such that Assumption \ref{assumption} holds. 

\begin{lemma}
	\label{l:topolDL1}
	$\mathcal{D}^{(M_{p})}_{L^{1}_{\omega}}$ is a quasinormable and thus distinguished Fr\'{e}chet space, and $\mathcal{D}^{\{M_{p}\}}_{L^{1}_{\omega}}$ is a complete and thus regular (LB)-space.
\end{lemma}

\begin{proof}
 To verify that  $\mathcal{D}^{(M_{p})}_{L^{1}_{\omega}}$ is quasinormable, it suffices to show that \cite[Lemma 26.14]{M-V-FuncAnal} 
		\begin{gather*} 
			\forall \ell > 0 ~ \forall r > 0 ~ \forall \varepsilon \in (0,1] ~\exists R > 0 ~ \forall \varphi \in \mathcal{D}^{(M_{p})}_{L^{1}_{\omega}} \text{ with } \| \varphi \|_{\mathcal{D}^{M_{p}, \ell / H}_{L^{1}_{\omega}}} \leq 1 \\ \nonumber
			 \exists \psi \in \mathcal{D}^{(M_{p})}_{L^{1}_{\omega}} \text{ with } \| \psi \|_{\mathcal{D}^{M_{p}, r}_{L^{1}_{\omega}}} \leq R \text{ such that } \| \varphi - \psi \|_{\mathcal{D}^{M_{p}, \ell}_{L^{1}_{\omega}}} \leq \varepsilon.
		\end{gather*}
Let $\ell > 0$ and $\varphi \in \mathcal{D}^{(M_{p})}_{L^{1}_{\omega}} \text{ with } \| \varphi \|_{\mathcal{D}^{M_{p}, \ell / H}_{L^{1}_{\omega}}} \leq 1$ be arbitrary. Choose $\chi \in \mathcal{S}^{(M_p)}_{(A_p)}(\reals^d)$ with $\int_{\reals^d} \chi(x) dx =1$ and  put $\chi_{\varepsilon} =  \varepsilon^{-d} \chi(\cdot / \varepsilon)$ for $\varepsilon \in (0,1]$. For any $r > 0$ we have that
\begin{align*} 
\| \varphi * \chi_{\varepsilon} \|_{\mathcal{D}^{M_{p}, r}_{L^{1}_{\omega}}} &= \sup_{\alpha \in \naturals^{d}} \frac{1}{r^{|\alpha|} M_{\alpha}} \int_{\reals^{d}} |\varphi * \chi_{\varepsilon}^{(\alpha)} (x)| \omega(x) dx \\
&\leq C \| \varphi \|_{L^{1}_{\omega}} \sup_{\alpha \in \naturals^{d}} \frac{1}{(\varepsilon r)^{|\alpha|} M_{\alpha}} \int_{ \reals^d} |\chi^{(\alpha)}(x)| e^{A(\lambda \varepsilon x)} dx \\  
&\leq C \|\chi\|_{\mathcal{S}^{M_p,\varepsilon r}_{A_p,\lambda H^{d+1}}}\int_{\reals^d} e^{A(\lambda x)-A( \lambda H^{d+1}x)} dx
.
\end{align*}
 On the other hand, applying the mean-value theorem, we obtain that
 		\begin{align*}
		&\| \varphi- \varphi * \chi_{\varepsilon} \|_{\mathcal{D}^{M_{p}, l}_{L^{1}_{\omega}}}\\
			&=  \sup_{\alpha \in \naturals^{d}} \frac{1}{\ell^{|\alpha|} M_{\alpha}} \| \varphi^{(\alpha)} - \varphi^{(\alpha)} * \chi_{\varepsilon} \|_{L^{1}_{\omega}} \\
			& \leq \varepsilon \sup_{\alpha \in \naturals^d}\frac{1}{\ell^{|\alpha|} M_{\alpha}}  \int_{\reals^{d}} \omega(x) \left( \int_{\reals^{d}} |\chi(t)| |t| \sum_{j = 1}^{d} \int_{0}^{1} | \varphi^{(\alpha + e_{j})}(x - \gamma \varepsilon t)| d\gamma dt \right) dx \\
			& \leq \left(  C C_{0} \frac{\ell}{H} d \int_{ \reals^d} \chi(t)|t| e^{A(\lambda t)} dt \right) \varepsilon ,
		\end{align*}
	from which the result easily follows. The completeness of $\mathcal{D}^{\{M_{p}\}}_{L^{1}_{\omega}}$ can be shown in a similar way as in \cite[Proposition 5.1]{dimovski2016}.
\end{proof}

We will need the ensuing basic density property.

	\begin{proposition}\label{proposition density}
		We have the following dense continuous inclusions,
		$$
		\sss^{\beurou}_{\dagger}(\reals^{d}) \hookrightarrow \mathcal{D}^{\beurou}_{L^{1}_{\omega}} \hookrightarrow \sss^{\prime \beurou}_{\dagger}(\reals^{d}).
		$$
	\end{proposition}
	
	\begin{proof} We adapt the idea from \cite[Proof of Proposition 5.2]{dimovski2016}.  It is clear that $\sss^{\beurou}_{\dagger}(\reals^{d}) \subset \mathcal{D}^{\beurou}_{L^{1}_{\omega}} \subset \sss^{\prime \beurou}_{\dagger}(\reals^{d})$ with continuous inclusions. Since $\sss^{\beurou}_{\dagger}(\reals^{d})$ is dense in $\sss^{\prime \beurou}_{\dagger}(\reals^{d})$, it suffices to show that $\sss^{(M_p)}_{(A_p)}(\reals^{d})$ is dense in  $\mathcal{D}^{\beurou}_{L^{1}_{\omega}}$. Choose $\chi \in \sss^{(M_{p})}_{(A_{p})}(\reals^{d})$ and $\psi\in\mathcal{D}(\reals^{d})$ such that $\int_{\reals^{d}} \chi(x)dx = 1$ and $\psi(0) = 1$. Next, set $\chi_{n} = n^{d} \chi(n \:\cdot\:)$ and $\psi_{n} = \psi(\:\cdot\: / n)$ for $n \geq 1$. Let $\ell > 0$ and $\varphi \in \mathcal{D}^{M_{p}, \ell}_{L^{1}_{\omega}}$ be arbitrary. We also fix an arbitrary $\varepsilon>0$. In view of the  inequality $e^{A(x+y)}\leq e^{A(2x)+A(2y)}$, it is clear that $\varphi_{n, j} = \chi_{n} * (\psi_{j} \varphi) \in \sss^{(M_{p})}_{(A_{p})}(\reals^{d})$. We shall show that there are $n,j\in\mathbb{N}$ such that $\norm{\varphi-\varphi_{n,j}}_{\mathcal{D}^{M_{p}, H\ell}_{L^{1}_{\omega}}}\leq  \varepsilon$, which will complete the proof. Obviously,
			\begin{equation}
				\label{eq:densityStriangleineq}
				\norm{\varphi - \varphi_{n,j}}_{\mathcal{D}^{M_{p}, H\ell}_{L^{1}_{\omega}}} \leq \norm{\varphi - \chi_{n} \ast \varphi}_{\mathcal{D}^{M_{p}, H\ell}_{L^{1}_{\omega}}} + \norm{\chi_{n}\ast (\varphi-\psi_{j}\varphi)}_{\mathcal{D}^{M_{p}, H\ell}_{L^{1}_{\omega}}} . 
			\end{equation}
The last part of the proof of Lemma \ref{l:topolDL1} already gives an estimate  for the first term in the right-hand side of \eqref{eq:densityStriangleineq},
\[
			\norm{\varphi - \chi_{n} \ast \varphi}_{\mathcal{D}^{M_{p}, H\ell}_{L^{1}_{\omega}}} 
				\leq n^{-1}CC_0 \ell d  \int_{\reals^{d}} |\chi(t)| |t| e^{A(\lambda t)} dt \| \varphi\|_{\mathcal{D}^{M_{p}, \ell}_{L^{1}_{\omega}}}   \leq \frac{ \varepsilon}{2},
	\]
for sufficiently large $n$.
For such a fixed $n$, we now proceed to estimate the second term in the right-hand side of \eqref{eq:densityStriangleineq}. We have that	
			\begin{align*}
				\norm{\chi_{n}\ast (\varphi-\psi_{j}\varphi)}_{\mathcal{D}^{M_{p}, H\ell}_{L^{1}_{\omega}}} & \leq C\norm{\varphi-\psi_j\varphi}_{L^{1}_{\omega}} \sup_{\alpha\in\mathbb{N}^{d}}\frac{n^{|\alpha|}}{(H\ell)^{|\alpha|}M_{\alpha}} \int_{\reals^{d}}|\chi^{(\alpha)}(x)| e^{A(\lambda x)}d x\\
				&\leq C\norm{\varphi-\psi_j\varphi}_{L^{1}_{\omega}} \norm{\chi}_{\mathcal{S}^{M_p, H\ell/n}_{A_p, \lambda H^{d+1}}}\int_{\reals^{d}}e^{A(\lambda x)-A(\lambda H^{d+1}x)}d x
				\leq \frac{\varepsilon}{2},
			\end{align*}
for large enough $j$.
	\end{proof}
The strong dual of $\mathcal{D}^{\beurou}_{L^{1}_{\omega}}$ is denoted by $\mathcal{B}^{\prime \beurou}_{\omega}$. By the previous proposition, we may view  $\mathcal{B}^{\prime \beurou}_{\omega}$ as a subspace of $\sss^{\prime \beurou}_{\dagger}(\reals^{d})$. We define $\dot{\mathcal{B}}^{\prime \beurou}_{\omega}$ as the closure in $\mathcal{B}^{\prime \beurou}_{\omega}$ of the space  of compactly supported continuous functions on $\reals^d$. Notice that $\dot{\mathcal{B}}^{\prime \beurou}_{\omega}$ coincides with the closure in $\mathcal{B}^{\prime \beurou}_{\omega}$ of $\sss^{(M_p)}_{(A_p)}(\reals^{d})$.

\subsection{Characterization via the STFT}

The goal of this subsection is to characterize $\mathcal{D}^{\beurou}_{L^{1}_{\omega}}$ and $\dot{\mathcal{B}}^{\prime \beurou}_{\omega}$ in terms of the STFT. We  first consider $\mathcal{D}^{\beurou}_{L^{1}_{\omega}}$.  
The following two lemmas are needed in our analysis.

	\begin{lemma}
		\label{l:Dl1STFTgrowth}
		Let $\psi \in \sss^{(M_{p})}_{(A_{p})}(\reals^{d})$. Then, for any $\ell > 0$, there is $C' > 0$ such that
			\[ \norm{V_{\psi} \varphi(\:\cdot\:, \xi)}_{L^{1}_{\omega}} \leq C' \norm{\varphi}_{\mathcal{D}^{M_{p}, \ell}_{L^{1}_{\omega}}} e^{-M(\pi \xi / (\ell \sqrt{d}))}, \qquad \xi \in \reals^d,  \]
			for all $\varphi \in \mathcal{D}^{M_{p}, \ell}_{L^{1}_{\omega}}$. 
	\end{lemma}
	
	\begin{proof}
		Let $\varphi \in \mathcal{D}^{M_{p}, \ell}_{L^{1}_{\omega}}$ be arbitrary. For any $\alpha \in \naturals^{d}$ we have that
			\begin{align*}
				&\int_{\reals^{d}} |\xi^{\alpha} V_{\psi} \varphi(x, \xi)| \omega(x) dx \\
				&\leq (2 \pi)^{-|\alpha|}\sum_{\beta \leq \alpha} {\alpha \choose \beta} \int_{\reals^{d}} \omega(x) \left( \int_{\reals^{d}} |\varphi^{(\beta)}(t)| |\psi^{(\alpha - \beta)}(x - t)| dt \right) dx \\
				&\leq C(2 \pi)^{-|\alpha|} \sum_{\beta \leq \alpha} {\alpha \choose \beta} \int_{\reals^{d}} |\varphi^{(\beta)}(t)| \omega(t) \left( \int_{\reals^{d}} |\psi^{(\alpha - \beta)}(x - t)| e^{A(\lambda(x - t))} dx \right) dt \\
				&\leq   C' \norm{\varphi}_{\mathcal{D}^{M_{p}, \ell}_{L^{1}_{\omega}}} (\ell / \pi)^{|\alpha|} M_{\alpha},
			\end{align*}
		whence
	\begin{align*}
	\int_{\reals^{d}} |V_{\psi} \varphi(x, \xi)| \omega(x) dx &\leq M_0C' \norm{\varphi}_{\mathcal{D}^{M_{p}, \ell}_{L^{1}_{\omega}}} \inf_{\alpha \in \naturals^{d}} \left(\frac{\ell \sqrt{d}}{\pi |\xi|}\right)^{|\alpha|} \frac{M_{\alpha}}{M_{0}} \\
	&= M_0C' \norm{\varphi}_{\mathcal{D}^{M_{p}, \ell}_{L^{1}_{\omega}}} e^{-M(\pi \xi / (\ell \sqrt{d}))} . 
	\end{align*}
	\end{proof}
	
	\begin{lemma}
		\label{l:Dl1STFTConj}
		Let $\psi \in \sss^{(M_{p})}_{(A_{p})}(\reals^{d})$. Suppose that $F$ is a measurable function on $\reals^{2d}$ such that
			\[ \sup_{\xi \in \reals^{d}} e^{M(4 \pi H^{d + 1} \xi / \ell)} \int_{\reals^{d}} |F(x, \xi)| \omega(x) dx < \infty . \]
		Then, the function
			\[ t \mapsto \int \int_{\reals^{2d}} F(x, \xi) M_{\xi} T_{x} \psi(t) dx d\xi \]
		belongs to $\mathcal{D}^{M_{p}, \ell}_{L^{1}_{\omega}}$.
	\end{lemma}
	
	\begin{proof}
		  For any $\alpha \in \naturals^{d}$ we have that
			\begin{align*} 
				\left| \partial^{\alpha}_{t} [ M_{\xi} T_{x} \psi(t)] \right| &\leq \sum_{\beta \leq \alpha} {\alpha \choose \beta} (2 \pi |\xi|)^{|\beta|} |\psi^{(\alpha - \beta)}(t - x)| \\ 
				&\leq \norm{\psi}_{\sss^{M_{p}, \ell / 2}_{A_{p}, \lambda H^{d+1}}} \ell^{|\alpha|} M_{\alpha} \exp[M(4\pi \xi / \ell) - A(\lambda H^{d+1}(t - x))] . 
		\end{align*}
		Hence, 
			\begin{align*}
				&\norm{\int \int_{\reals^{2d}} F(x, \xi) M_{\xi} T_{x} \psi dx d\xi}_{\mathcal{D}^{M_{p}, \ell}_{L^{1}_{\omega}}} \\
				&\leq \norm{\psi}_{\sss^{M_{p}, \ell / 2}_{A_{p}, \lambda H^{d+1}}} \int_{\reals^{d}} \omega(t) \left( \int \int_{\reals^{2d}} |F(x, \xi)| \exp[M(4\pi \xi / \ell) - A(\lambda H^{d+1} (t - x))] dx d\xi\right) dt \\
				&\leq C \norm{\psi}_{\sss^{M_{p}, \ell / 2}_{A_{p}, \lambda H^{d+1}}} \int_{\reals^{d}} e^{M(4\pi \xi / \ell)} \left( \int_{\reals^{d}} |F(x, \xi)| \omega(x) \left( \int_{\reals^{d}} e^{A(\lambda(t - x)) - A(\lambda H^{d+1}(t - x))} dt \right) dx \right) d\xi \\
				&\leq C' \int_{\reals^{d}} e^{M(4\pi \xi / \ell) - M(4\pi H^{d+1} \xi / \ell)} d\xi.
			\end{align*}	
	\end{proof}
	
We are now able to characterize $\mathcal{D}^{\beurou}_{L^{1}_{\omega}}$ via the STFT.

	\begin{proposition}
		\label{p:Dl1STFTchar}
		Let $\psi \in \sss^{(M_{p})}_{(A_{p})}(\reals^{d}) \setminus \{0\}$ and let  $f \in \sss^{\prime \beurou}_{\dagger}(\reals^{d})$. Then, $f \in \mathcal{D}^{\beurou}_{L^{1}_{\omega}}$ if and only if
			\begin{equation}
				\label{eq:Dl1STFTchar}
				\forall q > 0 ~ (\exists q > 0) \, : \, \sup_{\xi \in \reals^{d}} e^{M(q \xi)} \norm{V_{\psi} f(\:\cdot\:, \xi)}_{L^{1}_{\omega}} < \infty .
			\end{equation}
		If $B \subset \mathcal{D}^{\beurou}_{L^{1}_{\omega}}$ is a bounded set, then \eqref{eq:Dl1STFTchar} holds uniformly over $B$.
	\end{proposition}
	
	\begin{proof}
		The direct implication and the fact that  \eqref{eq:Dl1STFTchar} holds uniformly over bounded sets follows immediately from Lemma \ref{l:Dl1STFTgrowth} (and, in the Roumieu case, Lemma \ref{l:topolDL1}). Conversely, suppose that \eqref{eq:Dl1STFTchar} holds and choose $\gamma \in \sss^{(M_{p})}_{(A_{p})}(\reals^{d})$ such that $(\gamma, \psi)_{L^{2}} = 1$. By \eqref{eq:desing}, we have that, for all $\varphi \in \sss^{\beurou}_{\dagger}(\reals^{d})$, 
			\begin{align*}
				\ev{f}{\varphi} &= \int \int_{\reals^{2d}} V_{\psi} f(x, \xi) V_{\overline{\gamma}} \varphi(x, -\xi) dx d\xi \\
				&= \int \int_{\reals^{2d}} V_{\psi} f(x, \xi) \left( \int_{\reals^{d}} \varphi(t) M_{\xi} T_{x} \gamma(t) dt \right) dx d\xi \\
				&= \int_{\reals^d} \left( \int \int_{\reals^{2d}} V_{\psi} f(x, \xi) M_{\xi} T_{x} \gamma(t) dx d\xi \right) \varphi(t) dt ,
			\end{align*}
		where the switching of the integrals in the last step is permitted because of \eqref{eq:Dl1STFTchar}. Hence,
			\[ f = \int \int_{\reals^{2d}} V_{\psi} f(x, \xi) M_{\xi} T_{x} \gamma dx d\xi \]
		and we may conclude that $f \in \mathcal{D}^{\beurou}_{L^{1}_{\omega}}$ by applying Lemma \ref{l:Dl1STFTConj} to $F = V_{\psi} f$.
	\end{proof}
	
Next, we treat $\dot{\mathcal{B}}^{\prime \beurou}_{\omega}$. We again need some preparation. We denote by $C_{\omega}(\reals^d)$ the Banach space consisting of all $f \in C(\reals^d)$ such that $\|f\|_{L^\infty_\omega} < \infty$ and by  $C_{0,\omega}(\reals^d)$ its closed subspace consisting of all elements $f$ such that $\lim_{|x| \rightarrow \infty} f(x) / \omega(x) = 0$.  We endow  $C_{0,\omega}(\reals^d)$  with the norm $\norm{\:\cdot \:}_{L^{\infty}_{\omega}}$. The dual of $C_{0,\omega}(\reals^d)$ is denoted by $\mathcal{M}^1_{\omega}$.  For every $\mu \in \mathcal{M}^{1}_{\omega}$ there is a unique regular complex Borel measure $\nu \in \mathcal{M}^{1} = (C_0(\reals^d))'$ such that
$$
\ev{\mu}{ \varphi} =  \int_{\reals^d} \frac{\varphi(x)}{\omega(x)} d\nu(x), \qquad \varphi \in C_{0,\omega}(\reals^d).
$$
Moreover, $\| \mu \|_{\mathcal{M}^{1}_{\omega}} =  \| \nu \|_{\mathcal{M}^{1}} = |\nu|(\reals^d).$
By \cite[Theorem 6.13]{Rudin}, the natural inclusion $L^1_\omega \subset \mathcal{M}^{1}_{\omega}$ holds topologically, that is,
\begin{equation}
		\label{eq:omegaintegralequiv}
		\| \varphi \|_{L^1_\omega} = \sup_{f \in B_{C_{0, \omega}}} \left| \int_{\reals^{d}} \varphi(x) f(x) dx \right| , \qquad \varphi \in L^{1}_{\omega},
\end{equation}
where $B_{C_{0, \omega}}$ denotes the unit ball in $C_{0, \omega}(\reals^d)$. We define
$$
C_{(M_p)}(\reals^d) := \varinjlim_{q \to \infty}  C_{e^{M(q\: \cdot \:)}}(\reals^d), \qquad C_{\{M_p\}}(\reals^d) := \varprojlim_{q \to 0^+}  C_{e^{M(q\: \cdot \:)}}(\reals^d).
$$
The following canonical isomorphisms of lcHs hold
$$
C_{\omega}(\reals^d_x) \widehat{\otimes}_\varepsilon C_{\{M_{p}\}} (\reals^{d}_\xi) \cong  \varprojlim_{q \to 0^+}  C_{\omega \otimes e^{M(q\: \cdot \:)}}(\reals^{2d}_{x,\xi})
$$
and
$$
C_{0,\omega}(\reals^d_x) \widehat{\otimes}_\varepsilon C_{\{M_{p}\}} (\reals^{d}_\xi) \cong  \varprojlim_{q \to 0^+}  C_{0, \omega \otimes e^{M(q\: \cdot \:)}}(\reals^{2d}_{x,\xi}).
$$
Similarly, in view of  \eqref{eq:M2'}, \cite[Theorem 3.1(d)]{B-M-S} and \cite[Theorem 3.7]{B-M-S} yield the following canonical isomorphisms of lcHs
$$
C_{\omega}(\reals^d_x) \widehat{\otimes}_\varepsilon C_{(M_{p})} (\reals^{d}_\xi) \cong  \varinjlim_{q \to \infty}  C_{\omega \otimes e^{M(q\: \cdot \:)}}(\reals^{2d}_{x,\xi})
$$
and
$$
C_{0,\omega}(\reals^d_x) \widehat{\otimes}_\varepsilon C_{(M_{p})} (\reals^{d}_\xi) \cong  \varinjlim_{q \to \infty}  C_{0, \omega \otimes e^{M(q\: \cdot \:)}}(\reals^{2d}_{x,\xi}).
$$
We are ready to establish the mapping properties of the STFT on $\dot{\mathcal{B}}^{\prime \beurou}_{\omega}$. 
	\begin{proposition}
		\label{p:contSTFTBdot}
		Let $\psi \in \sss^{(M_{p})}_{(A_{p})}(\reals^{d})$. The following mappings are continuous,
			\[ V_{\psi} : \dot{\mathcal{B}}^{\prime \ast}_{\omega} \rightarrow C_{0,\omega}(\reals^d_x) \widehat{\otimes}_\varepsilon C_{\ast} (\reals^{d}_\xi) \]
		and
			\[ V_{\psi}^{*} : C_{0,\omega}(\reals^d_x) \widehat{\otimes}_\varepsilon C_{\ast} (\reals^{d}_\xi) \rightarrow \dot{\mathcal{B}}^{\prime \ast}_{\omega}. \]
	\end{proposition}
	
	\begin{proof}
		We first consider $V_{\psi}$. It suffices to show that $V_{\psi} : \mathcal{B}^{\prime *}_{\omega} \rightarrow C_{\omega}(\reals^{d}_{x}) \widehat{\otimes}_{\varepsilon} C_{*}(\reals^{d}_{\xi})$ is continuous. In fact, as $C_{0,\omega}(\reals^d_x) \widehat{\otimes}_\varepsilon C_{\ast} (\reals^{d}_\xi)$ is a closed topological subspace of $C_{\omega}(\reals^d_x) \widehat{\otimes}_\varepsilon C_{\ast} (\reals^{d}_\xi)$, the result would then follow from Proposition \ref{p:contSTFT} and the inclusion $\sss_{(A_{p})}(\reals^{d}_{x}) \widehat{\otimes} \sss_{(M_{p})}(\reals^{d}_{\xi}) \subset C_{0,\omega}(\reals^d_x) \widehat{\otimes}_\varepsilon C_{\ast} (\reals^{d}_\xi)$. Since $ \mathcal{B}^{\prime *}_{\omega}$ is bornological (see Lemma \ref{l:topolDL1} in the Beurling case), it suffices to show that $V_\psi(B)$ is bounded in $C_{\omega}(\reals^{d}_{x}) \widehat{\otimes}_{\varepsilon} C_{*}(\reals^{d}_{\xi})$ for all bounded sets $B \subset \mathcal{B}^{\prime *}_{\omega}$. For some $\ell > 0$ (for all $\ell > 0$) it holds that $\sup_{f \in B} \sup_{\varphi \in A} |\langle f, \varphi \rangle | < \infty$ for all $A \subset \mathcal{D}^{\ast}_{L^{1}_{\omega}}$ bounded with respect to the norm $\| \, \cdot  \, \|_{\mathcal{D}^{M_{p}, \ell}_{L^{1}_{\omega}}}$. As 
		$$\{ e^{-M(4 \pi \xi / \ell)} \omega^{-1}(x) \overline{M_{\xi} T_{x} \psi} : (x, \xi) \in \reals^{2d} \} \subset \mathcal{D}^{\ast}_{L^{1}_{\omega}}$$
		 is bounded  with respect to $\| \, \cdot \,  \|_{\mathcal{D}^{M_{p}, \ell}_{L^{1}_{\omega}}}$, it follows that
			\[ \sup_{f \in B} \sup_{(x, \xi) \in \reals^{2d}} e^{-M(4 \pi \xi / \ell)} \omega^{-1}(x) |V_{\psi} f(x, \xi)| < \infty . \]	
Next, we treat $V^\ast_{\psi}$.  Lemma \ref{l:Dl1STFTgrowth} implies that  $V_{\psi}^{*} : C_{\omega}(\reals^d_x) \widehat{\otimes}_\varepsilon C_{\ast} (\reals^{d}_\xi) \rightarrow \mathcal{B}^{\prime \ast}_{\omega}$ is continuous. We claim that $\sss_{(A_{p})}(\reals^{d}_{x}) \widehat{\otimes} \sss_{(M_{p})}(\reals^{d}_{\xi})$ is dense in $C_{0,\omega}(\reals^d_x) \widehat{\otimes}_\varepsilon C_{\ast} (\reals^{d}_\xi)$, whence  the result follows from Proposition \ref{p:contSTFT}. We now prove the claim. It is clear that $\sss_{(A_{p})}(\reals^{d})$ is dense in $C_{0,\omega}(\reals^d)$ and that $\sss_{(M_{p})}(\reals^{d})$ is dense in $C_{\ast} (\reals^{d})$. Hence, the claim is a consequence of  the following general fact: Let $E,E_0,F,F_0$ be lcHs such that $E_0 \subseteq E$ and $F_0 \subseteq F$ with dense continuous inclusions. Then,  $E_0 \widehat{\otimes}_\varepsilon F_0$ is dense in $E \widehat{\otimes}_\varepsilon F$.	
	\end{proof}
	\begin{corollary}\label{lct}
	 $\dot{\mathcal{B}}^{\prime (M_p)}_{\omega}$ is a complete $(LB)$-space, and $\dot{\mathcal{B}}^{\prime \{M_p\}}_{\omega}$ is a quasinormable  Fr\'echet space.
	\end{corollary}
	\begin{proof}
	Proposition \ref{p:contSTFTBdot} and the reconstruction formula \eqref{eq:reconstructSTFT-1} imply that $\dot{\mathcal{B}}^{\prime \beurou}_{\omega}$ is isomorphic to a complemented subspace of $C_{0,\omega}(\reals^d) \widehat{\otimes}_\varepsilon C_{\ast} (\reals^{d})$. Hence, it suffices to notice that $C_{0,\omega}(\reals^d) \widehat{\otimes}_\varepsilon C_{(M_p)} (\reals^{d})$ is an $(LB)$-space  that is complete and $C_{0,\omega}(\reals^d) \widehat{\otimes}_\varepsilon C_{\{M_p\}} (\reals^{d})$ is a quasinormable Fr\'echet space by \cite[Proposition 2]{B-E}. 
	\end{proof}
	
Proposition \ref{p:contSTFTBdot} allows for the following characterization of $\dot{\mathcal{B}}^{\prime \beurou}_{\omega}$ via the STFT.

	\begin{theorem}
		\label{t:Bdotequiv}
		Let $\psi \in \sss^{(M_{p})}_{(A_{p})}(\reals^{d}) \setminus \{0\}$ and let $f \in \sss^{\prime \beurou}_{\dagger}(\reals^{d})$. The following statements are equivalent:
			\begin{enumerate}
				\item $f \in \dot{\mathcal{B}}^{\prime \beurou}_{\omega}$.
				\item $\lim_{|h| \rightarrow \infty} T_{-h} f / \omega(h) = 0$ in $\sss^{\prime \beurou}_{\dagger}(\reals^{d})$.
				\item For some $q > 0$ (for all $q > 0$) it holds that
					\begin{equation}
						\label{eq:STFTvanish}
						\lim_{|(x, \xi)| \rightarrow \infty} e^{-M(q \xi)} \frac{|V_{\psi} f(x, \xi)|}{\omega(x)} = 0 . 
					\end{equation}
				\end{enumerate}
	\end{theorem}
	
	\begin{proof}
		$(1) \Rightarrow (2)$: Since $\sss^{\beurou}_{\dagger}(\reals^{d})$ is Montel, it suffices to show that  $\lim_{|h| \rightarrow \infty} T_{-h} f / \omega(h) = 0$ weakly in $\sss^{\prime \beurou}_{\dagger}(\reals^{d})$. Take any $\varphi \in \sss^{\beurou}_{\dagger}(\reals^{d})$ and let $\varepsilon > 0$ be arbitrary. The set $\{ T_{h} \varphi / \omega(h) : h \in \reals^{d} \}$ is bounded in $\mathcal{D}^{\beurou}_{L^{1}_{\omega}}$. Hence, there is  $\chi \in \sss^{\beurou}_{\dagger}(\reals^{d})$ such that $|\ev{T_{-h}(f - \chi)}{\varphi}| \leq \varepsilon \omega(h)$ for all $h \in \reals^{d}$. We obtain that
			\[ \limsup_{|h| \rightarrow \infty} \frac{|\ev{T_{-h} f}{\varphi}|}{\omega(h)} \leq \varepsilon + \lim_{|h| \rightarrow \infty} \frac{1}{\omega(h)} \left| \int_{\reals^{d}} \varphi(t - h) \chi(t) dt \right| = \varepsilon . \]
			
		$(2) \Rightarrow (3)$: We only treat the Beurling case as the Roumieu case is similar. Since  the mapping
		\[ *_{f} : \sss^{(M_{p})}_{(A_{p})}(\reals^{d}) \rightarrow C_{(A_{p})}(\reals^{d}), \, \varphi \mapsto f * \varphi \]
		is  continuous and our assumption yields that $*_f( \sss^{(M_{p})}_{(A_{p})}(\reals^{d})) \subset C_{0,\omega}(\reals^d)$, we may infer from the closed graph theorem that  $*_{f} : \sss^{(M_{p})}_{(A_{p})}(\reals^{d}) \rightarrow C_{0, \omega}(\reals^d)$ is continuous. Hence, there is $\ell > 0$ such that $*_f$ can be uniquely extended to a  continuous linear mapping $*_{f} : \overline{\sss^{(M_{p})}_{(A_{p})}(\reals^{d})}^{\sss^{M_{p}, \ell}_{A_{p}, 1/\ell}(\reals^{d})} \rightarrow C_{0, \omega}(\reals^d)$. Fix $q' > 4\pi/l$. As $\{  e^{-M(q' \xi)}M_{\xi} \check{\overline{\psi}}  : \xi \in \reals^{d} \}$ is relatively compact in $\overline{\sss^{(M_{p})}_{(A_{p})}(\reals^{d})}^{\sss^{M_{p}, \ell}_{A_{p}, 1/\ell}(\reals^{d})}$, we obtain that
	$$
	\{e^{-M(q' \xi)}|V_\psi f(x,\xi)|  : \xi \in \reals^{d} \} = \{ \ast_f(e^{-M(q' \xi)} M_{\xi} \check{\overline{\psi}}) : \xi \in \reals^{d} \} 
	$$	
is relatively compact in $C_{0,\omega}(\reals^d)$. This implies that
$$
\lim_{|x| \rightarrow \infty}\sup_{\xi \in \reals^d} e^{-M(q' \xi)} \frac{|V_{\psi} f(x, \xi)|}{\omega(x)} = 0,
$$		
whence \eqref{eq:M2'} implies that \eqref{eq:STFTvanish} holds for any $q > q'$.

		$(3) \Rightarrow (1)$:  $(3)$ means that $V_{\psi} f \in C_{0,\omega}(\reals^d_x) \widehat{\otimes}_\varepsilon C_{\ast} (\reals^{d}_\xi)$. The result therefore follows from Proposition  \ref{p:contSTFTBdot} and the reconstruction formula \eqref{eq:reconstructSTFT-1}.		
	\end{proof}
	
In the non-quasianalytic case, we additionally have that:

	\begin{theorem}
		\label{t:BdotequivNQA}
		Let $f \in \mathcal{D}^{\prime \beurou}(\reals^{d})$. Then, $f \in \dot{\mathcal{B}}^{\prime \beurou}_{\omega}$ if and only if $\lim_{|h| \rightarrow \infty} T_{-h} f / \omega(h) = 0$ in $\mathcal{D}^{\prime \beurou}(\reals^{d})$.
	\end{theorem}
	
	\begin{proof}
		Necessity follows immediately from Theorem \ref{t:Bdotequiv}. To show sufficiency, we notice that $f \in \mathcal{B}^{\prime \beurou}_{\omega} \subset \sss^{\prime \beurou}_{\dagger}(\reals^{d})$ by \cite[Theorem 1]{N-V-asympboundMAE}. Next, one  may  obtain \eqref{eq:STFTvanish} for some $q > 0$ (for all $q > 0$) by taking a window function $\psi \in \mathcal{D}^{(M_{p})}(\reals^{d}) \setminus \{0\}$ and making minor adjustments in the proof of $(2) \Rightarrow (3)$ in Theorem \ref{t:Bdotequiv}. Hence, the result follows from Theorem \ref{t:Bdotequiv}.
	\end{proof}

\section{The structure of $\dot{\mathcal{B}}_{\omega}^{\prime \beurou}$}\label{sect-main}
The goal of this section is to prove Theorem \ref{t:structuraltheorem}. As before, we fix a weight sequence $M_p$ satisfying $(M.1)$ and $(M.2)'$, and a weight function $\omega$ such that Assumption \ref{assumption} holds. We will work with the following spaces of vector-valued multi-sequences.
Let $E$ be a Banach space. For  $\ell > 0$ we define  $\Lambda_{M_{p}, \ell}(E)$ as the Banach space consisting of all (multi-indexed) sequences $(e_{\alpha})_{\alpha \in \naturals^{d}} \in E^{\naturals^{d}}$ such that
	\[ \norm{(e_{\alpha})_{\alpha \in \naturals^d}}_{\Lambda_{M_{p}, \ell}(E)} := \sup_{\alpha \in \naturals^{d}} \ell^{|\alpha|} M_{\alpha} \|e_{\alpha}\|_E < \infty . \]
We define
	\[ \Lambda_{(M_{p})}(E) := \varinjlim_{\ell \rightarrow 0^+} \Lambda_{M_{p}, \ell}(E) , \qquad \Lambda_{\{M_{p}\}}(E) := \varprojlim_{\ell \rightarrow \infty} \Lambda_{M_{p}, \ell}(E). \]
$\Lambda_{(M_{p})}(E)$ is a complete $(LB)$-space by \cite[Theorem 2.6]{B-M-S}, and $\Lambda_{\{M_{p}\}}(E)$ is a Fr\'echet space.  Given a Banach space $F$, we set $\Lambda'_{(M_p)}(F): = \Lambda_{\{1 / M_{p}\}}(F)$ and $\Lambda'_{\{M_p\}}(F) := \Lambda_{(1 / M_{p})}(F)$. We then have the following canonical isomorphisms of lcHs
\[  (\Lambda_{(M_{p})}(E))' \cong  \Lambda'_{(M_p)}(E'), \qquad (\Lambda_{\{M_{p}\}}(E))' \cong  \Lambda'_{\{M_p\}}(E'). \]
Theorem \ref{t:structuraltheorem} may now be reformulated as follows.

	\begin{theorem}
		\label{t:structureasmap}
		The mapping
			\[ S : \Lambda_{\beurou}(C_{0, \omega}(\reals^d)) \rightarrow \dot{\mathcal{B}}^{\prime \beurou}_{\omega}, \quad (f_{\alpha})_{\alpha \in \naturals^{d}} \mapsto \sum_{\alpha \in \naturals^{d}} f^{(\alpha)}_{\alpha}  \]
		is surjective.
	\end{theorem}

Our plan is to show Theorem \ref{t:structureasmap} by applying Lemma \ref{t:surjectivity}. We need several preliminary results.

	\begin{lemma}
		$S$ is a well-defined continuous linear mapping.
	\end{lemma}
	
	\begin{proof}
		One easily verifies that $S : \Lambda_{\beurou}(C_{0, \omega}(\reals^d)) \rightarrow \mathcal{B}^{\prime \beurou}_{\omega}$ is a continuous linear mapping and that
		$\lim_{|h| \rightarrow \infty} T_{-h} S((f_{\alpha})_{\alpha \in \naturals^d}) / \omega(h) = 0$ in $\mathcal{S}^{\prime \ast}_{\dagger}(\reals^d)$ for all $(f_\alpha)_{\alpha \in \naturals^d} \in \Lambda_{\beurou}(C_{0, \omega}(\reals^d))$. Hence, the result follows from Theorem \ref{t:Bdotequiv}. 
	
	\end{proof}

Our next goal is to determine the transpose of $S$. To this end, we first show that, similarly as in the distributional case \cite{Schwartz}, the dual of $\dot{\mathcal{B}}^{\prime \beurou}_{\omega}$ is given by $\mathcal{D}^{\beurou}_{L^{1}_{\omega}}$.

	\begin{proposition}
		\label{p:dualBdot}
		The canonical inclusion mapping
			\[\iota: \mathcal{D}^{\beurou}_{L^{1}_{\omega}} \rightarrow (\dot{\mathcal{B}}^{\prime \beurou}_{\omega})^{\prime}_{b}, \, \quad \varphi \mapsto (f \mapsto \ev{f}{\varphi}) \]
		is a topological isomorphism.
	\end{proposition}
	
	\begin{proof}
		Clearly, $\iota$ is continuous and injective. Since  $\mathcal{D}^{\beurou}_{L^{1}_{\omega}}$ is webbed and  $(\dot{\mathcal{B}}^{\prime \beurou}_{\omega})^{\prime}_{b}$ is ultrabornological (Corollary \ref{lct}), it suffices, by De Wilde's open mapping theorem, to show that $\iota$ is surjective. Let $\Phi \in (\dot{\mathcal{B}}^{\prime \beurou}_{\omega})^{\prime}$ be arbitrary. Denote by $\rho : \sss^{\beurou}_{\dagger}(\reals^{d}) \rightarrow \dot{\mathcal{B}}^{\prime \beurou}_{\omega}$ the canonical inclusion and set $f = \Phi \circ \rho \in \sss^{\prime \beurou}_{\dagger}(\reals^{d})$. As $\Phi(\rho(\chi)) = \ev{f}{\chi}$ for every $\chi \in \sss^{\beurou}_{\dagger}(\reals^{d})$ and $\sss^{\beurou}_{\dagger}(\reals^{d})$ is dense in $\dot{\mathcal{B}}^{\prime \beurou}_{\omega}$, it is enough to show that $f \in \mathcal{D}^{\beurou}_{L^{1}_{\omega}}(\reals^{d})$. Let $\psi \in \sss^{(M_{p})}_{(A_{p})}(\reals^{d})$ be a fixed non-zero window function. Since $\Phi$ is continuous, there is a bounded set $B \subset \mathcal{D}^{\beurou}_{L^{1}_{\omega}}$ such that
			\[ |V_{\psi} f(x, \xi)| = |\Phi(\rho(\overline{M_{\xi} T_{x} \psi}))| \leq \sup_{\varphi \in B} |\ev{\overline{M_{\xi} T_{x} \psi}}{\varphi}| = \sup_{\varphi \in B} |V_{\psi} \varphi(x, \xi)| . \]
		Proposition \ref{p:Dl1STFTchar} implies that for every $q > 0$ (for some $q > 0$)
			\[ \sup_{\xi \in \reals^{d}} e^{M(q\xi)} \norm{V_{\psi} f(\:\cdot\:, \xi)}_{L^{1}_{\omega}} \leq \sup_{\varphi \in B} \sup_{\xi \in \reals^{d}} e^{M(q\xi)} \norm{V_{\psi} \varphi(\: \cdot \:, \xi)}_{L^{1}_{\omega}} < \infty , \]
		so that another application of Proposition \ref{p:Dl1STFTchar} shows that $f \in \mathcal{D}^{\beurou}_{L^{1}_{\omega}}$.
	\end{proof}
	
	\begin{corollary}
		\label{c:transposedmapspecific}
		The transposed mapping $S^{t}$ may be identified with the continuous linear mapping
			\[ \mathcal{D}^{\ast}_{L^{1}_{\omega}} \rightarrow \Lambda^{\prime}_{\beurou}(\mathcal{M}^{1}_{\omega}) : \varphi \mapsto ((-1)^{|\alpha|}\varphi^{(\alpha)})_{\alpha \in \naturals^{d}} . \]
	\end{corollary}

	\begin{proof}[Proof of Theorem \ref{t:structureasmap}]
		We shall show that $S$ is surjective via Lemma \ref{t:surjectivity}. The space $\Lambda_{\beurou}(C_{0, \omega}(\reals^d))$ is clearly Mackey, while $\Lambda_{\beurou}(C_{0, \omega}(\reals^d)) / \ker S$ is complete as $\Lambda_{\beurou}(C_{0, \omega}(\reals^d))$ is  complete. Next, we show that $\im S$ is Mackey. In the Romieu case this is trivial because $\dot{\mathcal{B}}^{\prime \{M_{p}\}}_{\omega}$ is a Fr\'echet space.  We now consider the Beurling case. We shall prove that $X = \im S$ is infrabarreled and thus Mackey. We need to show that every strongly bounded set $B$ in $X'$ is equicontinuous. Since $X$ is dense in $\dot{\mathcal{B}}^{\prime (M_{p})}_{\omega}$ (as $S^t$ is injective), Proposition \ref{p:dualBdot} implies that $X' = \mathcal{D}^{(M_{p})}_{L^{1}_{\omega}}$. 
		For arbitrary $\ell > 0$ we consider the set
			\[ V_{\ell} = \left\{ \frac{f^{(\alpha)}}{\ell^{|\alpha|} M_{\alpha}} : \alpha \in \naturals^{d} , f \in B_{C_{0, \omega}} \right\} \subseteq X . \]
	The set	$V_{\ell}$ is bounded in $X$ because $S$ is continuous, so that  $\sup_{\varphi \in B} \sup_{g \in V_{\ell}} |\ev{\varphi}{g}| < \infty$. The relation \eqref{eq:omegaintegralequiv} yields that
			\begin{align*}
				\sup_{\varphi \in B} \sup_{g \in V_{\ell}} |\ev{\varphi}{g}| &= \sup_{\varphi \in B} \sup_{\alpha \in \naturals^{d}} \sup_{f \in B_{C_{0, \omega}}}  \left| \ev{\varphi}{\frac{f^{(\alpha)}}{\ell^{|\alpha|} M_{\alpha}}} \right| \\
				&= \sup_{\varphi \in B} \sup_{\alpha \in \naturals^{d}} \frac{1}{\ell^{|\alpha|} M_{\alpha}} \sup_{f \in B_{C_{0, \omega}}} \left| \int_{\reals^{d}} \varphi^{(\alpha)}(x) f(x) dx \right| 
				= \sup_{\varphi \in B} \sup_{\alpha \in \naturals^{d}} \frac{\norm{\varphi^{(\alpha)}}_{L^{1}_{\omega}}}{\ell^{|\alpha|} M_{\alpha}}.
			\end{align*}
		Hence,
			\[ \sup_{\varphi \in B} \norm{\varphi}_{\mathcal{D}^{M_{p}, \ell}_{L^{1}_{\omega}}} < \infty, \qquad \forall \ell > 0 , \]
		which means that $B$ is bounded in $\mathcal{D}^{(M_{p})}_{L^{1}_{\omega}}$. Then, $B$ is equicontinuous because of Proposition \ref{p:dualBdot} and the fact that  $\dot{\mathcal{B}}^{\prime (M_{p})}_{\omega}$  is barreled (Corollary \ref{lct}). We already noticed that $S^t$ is injective. Finally, we show that $\im S^t$ is weakly closed in $\Lambda^{\prime}_{\beurou}(\mathcal{M}^{1}_{\omega})$.  Let $(\varphi_{j})_{j}$ be a net in $\mathcal{D}^{\beurou}_{L^{1}_{\omega}}$ and $(\mu_{\alpha})_{\alpha \in \naturals^{d}} \in \Lambda^{\prime}_{\beurou}(\mathcal{M}^{1}_{\omega})$  such that $((-1)^{|\alpha|}\varphi_{j}^{(\alpha)})_{\alpha \in \naturals^{d}} \rightarrow (\mu_{\alpha})_{\alpha \in \naturals^{d}}$ weakly in $\Lambda^{\prime}_{\beurou}(\mathcal{M}^{1}_{\omega})$. In particular, $\varphi^{(\alpha)}_{j} \rightarrow (-1)^{|\alpha|}\mu_{\alpha}$ weakly in $\mathcal{M}^{1}_{\omega}$ for all $\alpha \in \naturals^d$. Consequently, we have that $\mu_{0}^{(\alpha)} =(-1)^{|\alpha|} \mu_{\alpha} \in \mathcal{M}^{1}_{\omega}$ for all $\alpha \in \naturals^d$ (the derivatives should be interpreted in the sense of distributions). The equality \eqref{eq:omegaintegralequiv} implies that $\mu_{0} \in \mathcal{D}_{L^{1}_{\omega}}$ and that
			\[ (\| \mu^{(\alpha)}_{0}\|_{L^{1}_{\omega}})_{\alpha \in \naturals^{d}} = (\|\mu^{(\alpha)}_{0}\|_{\mathcal{M}^{1}_{\omega}})_{\alpha \in \naturals^{d}} = (\norm{\mu_{\alpha}}_{\mathcal{M}^{1}_{\omega}})_{\alpha \in \naturals^{d}} \in \Lambda^{\prime}_{\beurou}(\complexes),  \]
		which means that $\mu_{0} \in \mathcal{D}^{\beurou}_{L^{1}_{\omega}}$. Hence, $(\mu_{\alpha})_{\alpha \in \naturals^{d}} = ((-1)^{|\alpha|}\mu_{0}^{(\alpha)})_{\alpha \in \naturals^{d}}  \in \im S^t$. 
	\end{proof}

\section{The structure of S-asymptotics}\label{sect-strucS}
We now determine the structure of the  S-asymptotic behavior of ultradistributions, effectively a variant of \cite[Theorem 1.10, p.\ 46]{P-S-V}. 
Throughout this section, we fix a weight sequence $M_p$ satisfying $(M.1)$, $(M.2)'$, and $(M.3)'$. 

Let $\omega$ be a weight function. We consider a convex cone $\Gamma$ (with vertex at the origin). For $R>0$, we write $\Gamma_{R} = \Gamma + B(0, R)$. We will work with the following assumption on $\omega$: the limits
\begin{equation}\label{eq:cond S-asymptotic}
\underset{h\in \Gamma}{\lim_{|h|\to\infty}}\frac{\omega(x+h)}{\omega(h)} \qquad\mbox{exist for all}\ x\in\mathbb{R}^{d}.
\end{equation}
Then, an ultradistribution $f \in \mathcal{D}^{\prime \beurou}(\reals^{d})$ is said to have \emph{S-asymptotic behavior with respect to $\omega$ on $\Gamma$}, with limit $g \in \mathcal{D}^{\prime \beurou}(\reals^{d})$,  if
	\begin{equation} 
		\label{eq:Sasymp}
		\lim_{h \in \Gamma, |h| \rightarrow \infty} \frac{\ev{f(x + h)}{\varphi(x)}}{\omega(h)} = \ev{g(x)}{\varphi(x)} , \qquad \forall \varphi \in \mathcal{D}^{\beurou}(\reals^{d}) . 
	\end{equation}
If $g\neq 0$, one readily obtains that \eqref{eq:cond S-asymptotic} must hold uniformly for $x$ in compact subsets.

We now apply Theorem \ref{t:structuraltheorem} to find the structure of the S-asymptotic behavior of ultradistributions. 
	\begin{theorem}
		\label{t:structSasymp}
		Let $\Gamma\subset \mathbb{R}^{d}$ be a convex cone such that $\operatorname*{int}\Gamma$ is non-empty and let $\omega$ be a weight function satisfying \eqref{TIWFeq} and \eqref{eq:cond S-asymptotic}. Then, $f \in \mathcal{D}^{\prime \beurou}(\reals^{d})$ has S-asymptotic behavior with respect to $\omega$ on $\Gamma$ if and only if for each $R > 0$ there exist $f_{\alpha} \in C(\reals^d)$, $\alpha \in \naturals^{d}$, such that
			$$f = \sum_{\alpha \in \naturals^{d}} f_{\alpha}^{(\alpha)} \quad \text{on } \Gamma_{R} , $$
		 the limits
			$$ \underset{x \in \Gamma_{R}}{\lim_{ |x| \rightarrow \infty}} \frac{f_{\alpha}(x)}{ \omega(x)}, \qquad \alpha \in \naturals^{d}, $$
		exist, and for some $\ell > 0$ (for all $\ell > 0$) it holds that
			$$ 				 \sup_{\alpha \in \naturals^d, \ x \in \Gamma_{R}} \frac{\ell^{|\alpha|}{M_{\alpha} |f_{\alpha}(x)|}}{\omega(x)} < \infty.
$$
	\end{theorem}
	
	\begin{proof}
		The conditions are clearly sufficient. To show necessity, let us first verify that there is a constant $C$ such that $f_0= f - C \omega$ has S-asymptotic behavior with respect to $\omega$ on $\Gamma$ with limit 0. By \cite[Proposition~1.2, p.~12]{P-S-V}, there is $y\in\mathbb{R}^{d}$ such that the limits \eqref{eq:cond S-asymptotic} equal $e^{y\cdot x}$ for each $x\in\mathbb{R}^{d}$ and $g(x)=C e^{y\cdot x}$. Thus, $f_0$ with this $C$ satisfies the requirement. We further consider the case $\Gamma = \reals^{d}$, the general case can be reduced to this one by applying the same technique as in the proof of \cite[Theorem 3.2]{N-V-asympboundMAE}. Notice that $\omega$ is $(p!)$-admissible (see  Remark \ref{remark-1}). As $(M.1)$ and $(M.3)'$ imply that $p! \prec M_p$, we have that $M_p$ and $\omega$ satisfy Assumption \ref{assumption}. We obtain $f_{0} \in \dot{\mathcal{B}}^{\prime \beurou}_{\omega}$  by Theorem \ref{t:BdotequivNQA}. Hence, the desired structure of $f$ follows from Theorem \ref{t:structuraltheorem}. 
	\end{proof}
	
	\begin{remark} \label{remark-2}
		In \cite{structquasiultra}, the last two named authors obtained structural theorems for the so-called quasiaymptotic behavior of  ultradistributions upon reducing their analysis to the S-asymptotic behavior via an exponential substitution. Hence, as a direct consequence of Theorem \ref{t:structSasymp}, we obtain that the assumptions $(M.1)$, $(M.2)$, and $(M.3)$ in \cite{structquasiultra} can be everywhere relaxed to $(M.1)$, $(M.2)'$, and $(M.3)'$.
	\end{remark}


\begin{thebibliography}{99}

\bibitem{B-O2014} {\sc C.~Bargetz and N.~Ortner,} \emph{Characterization of L. Schwartz' convolutor and multiplier spaces $\mathcal{O}'_{C}$ and $\mathcal{O}_{M}$  by the short-time Fourier transform,} Rev. R. Acad. Cienc. Exactas F\'{i}s. Nat. Ser. A Mat. RACSAM 108 (2014),  833--847. 

\bibitem{bastin89}{\sc F.~Bastin}, {\em On bornological $C\overline{V}(X)$ spaces,} Arch. Math. 53 (1989), 394--398.

\bibitem{B-E} {\sc F.~Bastin and B.~Ernst},  \emph{A criterion for $CV(X)$ to be quasinormable}, Results Math. 14 (1988), 223--230.

\bibitem{B-F-G-ultradistrLpgrowth}
{\sc J.~J.~Betancor, C.~Fern\'{a}ndez and A.~Galbis}, {\em Beurling ultradistributions of $L^{p}$-growth}, 
J. Math. Anal. Appl. 279 (2003), 246--265.

\bibitem{B-M-S} {\sc K.~D.~Bierstedt, R.~Meise and W.~H.~Summers}, \emph{A projective description of
weighted inductive limits}, Trans. Amer. Math. Soc. 272 (1982), 107--160.

\bibitem{carmichael2007boundary}
{\sc R.~D.~Carmichael, A.~Kami{\'n}ski and S.~Pilipovi{\'c}}, {\em Boundary values and convolution in ultradistribution spaces},
Series on Analysis, Applications and Computation, 1, World Scientific Publishing Co. Pte. Ltd., Hackensack, NJ, 2007.

\bibitem{cioranescu1992} {\sc I.~Cioranescu}, {\em The characterization of the almost periodic ultradistributions of Beurling type,} Proc. Amer. Math. Soc. 116 (1992), 127--134. 


\bibitem{D-V-Nontrivialanalyfunc}
{\sc A.~Debrouwere and J.~Vindas}, {\em On the non-triviality of certain spaces of analytic functions. Hyperfunctions and ultrahyperfunctions of fast growth}, 
Rev. R. Acad. Cienc. Exactas F\'{i}s. Nat. Ser. A. Math. RASCAM 112 (2018), 473--508.

\bibitem{D-V-indlimultra}{\sc A.~Debrouwere and J.~Vindas}, {\em On weighted inductive limits of spaces of ultradifferentiable functions and their duals,} Math. Nachr. 292 (2019), 573--602.

\bibitem{D-V-ConvSTFT} 
{\sc A.~Debrouwere and J.~Vindas}, {\em Topological properties of convolutor spaces via the short time Fourier transform}, Preprint, arXiv:1801.09246.

\bibitem{dimovski2016}
{\sc P.~Dimovski, S.~Pilipovi{\'c}, B.~Prangoski and J.~Vindas}, {\em Convolution of ultradistributions and ultradistribution spaces associated to translation-invariant Banach spaces},
Kyoto J. Math. 56 (2016), 401--440. 

\bibitem{D-P-V2015} {\sc P.~Dimovski, B.~Prangoski and J.~Vindas}, {\em On a class of translation-invariant spaces of quasianalytic ultradistributions,} Novi Sad J. Math. 45 (2015), 143--175.

\bibitem{Gelfand-Shilov3}
{\sc I.~M.~Gel'fand and G.~E.~Shilov}, {\em Generalized functions, Vol. 3},
Academic Press, New York, San Francisco, London, 1967.

 \bibitem{GelfandShilov}
{\sc I.~M.~Gel'fand and G.~E.~Shilov}, {\em Generalized functions, Vol. 2},
Academic Press, New York, San Francisco, London, 1968.


\bibitem{Grochenig}
{\sc K.~Gr{\"o}chenig}, {\em  Foundations of time-frequency analysis}, 
Birkh{\"a}user Boston, Inc., Boston, MA, 2001.

\bibitem{ultradistributions1}
{\sc H.~Komatsu}, {\em Ultradistributions. {I}. Structure theorems and a
  characterization}, J. Fac. Sci. Univ. Tokyo Sect. IA Math. 20 (1973), 25--105.
  

\bibitem{K-P-S-V2016} {\sc S.~Kostadinova, S.~Pilipovi\'{c}, K.~Saneva and J.~Vindas,} {\em The short-time Fourier transform of distributions of exponential type and Tauberian theorems for S-asymptotics,} Filomat 30 (2016), 3047--3061. 

\bibitem{M-V-FuncAnal}
{\sc R.~Meise, D.~Vogt}, {\em Introduction to functional analysis}, Clarendon Press, Oxford, 1997.
  
\bibitem{structquasiultra}
{\sc L.~Neyt and J.~Vindas}, {\em Structural theorems for quasiasymptotics of ultradistributions}, Asymptot. Anal. 114 (2019), 1--18. 
  
\bibitem{N-VMultDimTaub}
{\sc L.~Neyt and J.~Vindas}, {\em A multidimensional Tauberian theorem for Laplace transforms of ultradistributions}, Integral Transforms Spec. Funct., doi:10.1080/10652469.2019.1699556.

\bibitem{N-V-asympboundMAE}
{\sc L.~Neyt and J.~Vindas}, {\em Asymptotic boundedness and moment asymptotic expansion in ultradistribution spaces}, Appl. Anal. Discrete Math., to appear.

\bibitem{N-O}{\sc E.~A.~Nigsch and N.~Ortner}, \emph{The space $\dot{\mathcal{B}}'$ of distributions vanishing at infinity - duals of tensor products}, Rev. R. Acad. Cienc. Exactas F\'{i}­s. Nat., Ser. A Mat. 112 (2017), 251--269.


\bibitem{Ortner-1} {\sc N.~Ortner}, \emph{Sur la convolution des distributions}, C. R. Acad. Sci. Paris S\'er. A-B 290 (1980), A533--A536.

\bibitem{Ortner} {\sc N.~Ortner}, \emph{On convolvability conditions for distributions}, Monatsh. Math. 160 (2010), 313--335.

\bibitem{P-CharBoundedUltradistr}
{\sc S.~Pilipovi\'{c}}, {\em Characterizations of bounded sets in spaces of ultradistributions}, 
Proc. Amer. Math. Soc. 120 (1994), 1191--1206.
  
  
\bibitem{P-S-V} {\sc S.~ Pilipovi\'c, B.~Stankovi\'c and J.~Vindas}, \emph{Asymptotic behavior of generalized functions}, Series on Analysis, Applications and Computation, 5, World Scientific Publishing Co. Pte. Ltd., Hackensack, NJ, 2012.
 
 
\bibitem{Rudin} {\sc W.~Rudin}, \emph{Real and complex analysis}, Tata McGraw-Hill Education, 1987.


\bibitem{S-TVS}
{\sc H.~Schaefer}, {\em Topological vector spaces}, 
Graduate Texts in Mathematics, Springer-Verlag, New York, second edition, 1999.

\bibitem{Schwartz}
{\sc L. Schwartz}, {\em Th\'eorie des distributions}, 
Hermann, Paris, 1966.
  
\bibitem{Treves} {\sc F.~Tr\`eves}, {\em Topological vector spaces, distributions and kernels},
Academic Press, New York, 1967.

\end{thebibliography}
\end{document}